\documentclass[]{amsart}
\usepackage{mathrsfs}
\usepackage{graphicx}
\usepackage{amsmath}
\usepackage{galois, enumitem}
\usepackage[all,cmtip]{xy}
\usepackage{amsfonts,latexsym,rawfonts,amsmath,amssymb, amsthm}
\usepackage{latexsym,lscape,rawfonts}
\usepackage[all,cmtip]{xy}
\usepackage{extarrows,color}

\usepackage{times}

\allowdisplaybreaks
\usepackage{float}
\restylefloat{figure}
	\usepackage{caption2}
	\usepackage{array}
	\newtheorem{thm}{Theorem}[section]

	\newtheorem{cor}[thm]{Corollary}
	\newtheorem{lem}[thm]{Lemma}
	\newtheorem{prop}[thm]{Proposition}
	
	\theoremstyle{definition}
	\newtheorem{defn}[thm]{Definition}
	\newtheorem{exam}[thm]{Example}
	\newtheorem{ques}[thm]{Question}
	\theoremstyle{remark}
	\newtheorem{rem}[thm]{Remark}
	\numberwithin{equation}{section}

    \newcommand{\N}{\mathbb N}
	\newcommand{\Z}{\mathbb Z}

	\newcommand{\fix}{\mathrm{Fix}}

	\newcommand{\rk}{\mathrm{rk}}
\newcommand{\krk}{\mathrm{Krk}}
\newcommand{\ufgfp}{\mathrm{UFGFP}}
\newcommand{\id}{\mathrm{id}}

\newcommand{\fgfp}{\mathrm{FGFP}}
\newcommand{\fgfpm}{\mathrm{FGFP_m}}
\newcommand{\fgfpe}{\mathrm{FGFP_e}}
\newcommand{\fgfpa}{\mathrm{FGFP_a}}
\newcommand{\fp}{\mathrm{FP}}
\newcommand{\ufp}{\mathrm{UFP}}

\newcommand{\duf}{\mathrm{\mathfrak{D}_{uf}}}
\newcommand{\df}{\mathrm{\mathfrak{D}_{f}}}

	\newcommand{\aut}{\mathrm{Aut}}
	\newcommand{\edo}{\mathrm{End}}
	
\newcommand{\mon}{\mathrm{Mon}}

\begin{document}
	
\title{Explicit bounds for fixed subgroups of endomorphisms of free products}

\author{Jialin Lei}
\address{School of Mathematics and Statistics, Xi'an Jiaotong University, Xi'an 710049, CHINA}
\email{leijialin0218@stu.xjtu.edu.cn}

\author{Qiang Zhang}
\address{School of Mathematics and Statistics, Xi'an Jiaotong University, Xi'an 710049, CHINA}

\email{zhangq.math@mail.xjtu.edu.cn}

\date{\today}

\subjclass[2010]{20F65, 20F34, 57M07}

\keywords{Fixed subgroups, free products, Gromov hyperbolic groups, surface groups}

\thanks{The second author is partially supported by NSFC (Nos. 11961131004, 11971389 and 12271385).}
		
\begin{abstract}
For an automorphism $\phi$ of a free group $F_n$ of rank $n$, Bestvina and Handel showed that the rank $\rk(\fix(\phi))$ of the fixed subgroup is not greater than $n$ (the so-called Scott conjecture). Soon after Bestvina and Handel's announcement, their result was generalized by many authors in various directions. In this paper, we are interested in the fixed subgroups of endomorphisms of free products, focusing on explicit bounds for their ranks.
By introducing a concept of UFP degree and combining some results of O'Niell-Turner, Sela and Sykiotis, we get bounds for fixed subgroups of stably hyperbolic groups.
\end{abstract}
\maketitle

\section{Introduction}
For a finitely generated group $G$, the \emph{rank} of $G$ denoted $\rk(G)$ is the minimal number of generators of $G$. The rank $\rk(H)$ of a subgroup group $H$ of an abelian group $G$ can not be greater than $\rk(G)$. However, in general, the statement is false even in free groups. It is easy to see that the free group $F_n$ of rank $n$ is a subgroup of $F_2$ for all $n\in \N$.
Denote the monoid of endomorphisms (resp. monomorphisms, i.e. injective endomorphisms) of $G$ by $\edo(G)$ (resp. $\mon(G)$), and the group of automorphisms of $G$ by $\aut(G)$. For an endomorphism $\phi\in \edo(G)$, the \emph{fixed subgroup} of
$\phi$ is defined to be
 $$\fix(\phi) :=\{g\in G \mid \phi(g)=g\},$$
which is a subgroup of $G$ with many special properties.

For a free group $F_n$ of rank $n$, in 1975, Dyer and Scott \cite{DS75} proved that for a finite order automorphism $\phi$ of $F_n$, the rank $\rk(\fix(\phi))$ of the fixed subgroup is not greater than $n$. Moreover, Scott conjectured that $\rk(\fix(\phi))\leq n$ for any $\phi\in \aut(F_n)$, which is the so-called Scott conjecture. Once the conjecture was put forward, it attracted a lot of research, see \cite{Ve02} for a survey. Finally, in 1989, the conjecture was solved by Bestvina and Handel \cite{BH92}. Almost simultaneously, by introducing an important concept of \emph{stable image} (see Section \ref{Section. 2} for more details), Imrich and Turner \cite{IT89} reduced the fixed subgroups of endomorphisms to that of automorphisms, and then showed

\begin{thm}[Imrich-Turner, \cite{IT89}]\label{Scott conj for free gp endo}
If $\phi\in \edo(F_n)$, then $\rk(\fix(\phi))\leq n$.
\end{thm}

Soon after Bestvina and Handel's announcement, their result was generalized by many authors in various directions (for example, \cite{CT88, CT94, DV96, HW04, IT89, Pa89, Sh99} etc.). Moreover, it is necessary to mention some other recent articles and preprints which are relevant to fixed subgroups: inertness of fixed subgroups \cite{AJZ22, RV21, WZ}, algebraic applications of fixed subgroups \cite{Log22, Ven21}, algorithms for computing the rank of fixed subgroups \cite{BM16, CL22, Mut22}, fixed subgroups from topological viewpoint \cite{FH18, Lym22}, fixed subgroups of direct products \cite{Car22, WVZ, ZVW}, equalisers (more general structures than fixed subgroups) of homomorphisms \cite{Cio22, Log21}, and relationship between fixed subgroups and Nielsen fixed point theory \cite{ZZ1, ZZ2}.

For surface groups, a lot is known too (see \cite{JS, JWZ11, N2}). In this paper, a \emph{surface group} is the fundamental group $\pi_1(S)$ of a closed (orientable or not) surface $S$ with Euler characteristic $\chi(S)<0$. In \cite{JWZ11}, the authors showed

\begin{thm}[Jiang-Wang-Zhang, \cite{JWZ11}]\label{scott conj for surface gp}
Suppose $G$ is a surface group. Then for any endomorphism $\phi\in\edo(G)$,
$\rk(\fix(\phi))\leq \rk(G)$ with equality if and only if $\phi=\id$,
and $\rk(\fix(\phi))\leq \frac{1}{2} \rk(G)$ if $\phi$ is not epimorphic.
\end{thm}

To our knowledge, there are only a few results of this type for the fundamental group of a 3-manifold, see \cite{JWWZ21, LW14, Zhang1, Zhang2}.
In particular, Lin and Wang \cite{LW14} investigated the fundamental group $\pi_1(M)$ of a \emph{hyperbolic 3-manifold} $M$, i.e., $M$ is compact, orientable, and the interior of $M$ admits a complete hyperbolic structure of finite volume. For later use, we list their result in the following.

\begin{thm}[Lin-Wang, \cite{LW14}]\label{lin-wang hyper 3 mfd}
Suppose $\phi$ is an automorphism of $G =\pi_1(M)$, where M is a hyperbolic 3-manifold. Then $\rk(\fix(\phi)) < 2 \rk(G).$
\end{thm}

In this paper, we are mainly interested in the fixed subgroups of free products, focusing on explicit bounds for their ranks.

For a group $G$, we say that $G$ has the \emph{finitely generated fixed subgroup property} of monomorphisms (resp. automorphisms, endomorphisms), abbreviated as $\fgfpm$ (resp. $\fgfpa$, $\fgfpe$), if for any $f\in \mon(G)$ (resp. $\aut(G)$, $\edo(G)$), the fixed subgroup $\fix(f)$ is finitely generated. Clearly, if $G$ has $\fgfpe$, then it has $\fgfpm$ and hence has $\fgfpa$. Theorem \ref{Scott conj for free gp endo} and \ref{scott conj for surface gp} imply that free groups and surface groups have $\fgfpe$. Furthermore, Lemma \ref{uf index} shows that the free groups $F_2$ and $\Z$ both have $\fgfpa$ but their direct product don't. Therefore, it is natural to ask

\begin{ques}\label{Question 1}
Is $\fgfpm$ ($\fgfpa$ or $\fgfpe$) preserved under taking free product? Moreover, if the answer is affirmative, what is the quantitative
relation among the explicit ranks of fixed subgroups of the free product $\ast_{i=1}^n G_i$ and the factor groups $G_1, G_2, \ldots, G_n$?
\end{ques}

For the cases $\fgfpm$ and $\fgfpa$, we have an affirmative answer to Question \ref{Question 1}, see \textbf{Theorem \ref{main thm 0}} (and Corollary \ref{hyper 3 mfd gp} on 3-manifold groups). Furthermore, to quantitatively analysis the ranks of fixed subgroups of a group $G$, we introduce a concept of $\ufp ~degree$, denoted $\duf(G)$ (see Definition \ref{def for UFP}) and show some explicit bounds for the fixed subgroups, see \textbf{Proposition \ref{Main Lemma}} and \textbf{Proposition \ref{key prop}} in Section \ref{Section. 3}.

For $\fgfpe$, we can give affirmative answers for some special kinds of groups. Note that Paulin \cite{Pa89} proved that Gromov hyperbolic groups have $\fgfpa$ (see \cite{HW04, Neu92, Sh99} for more information on fixed subgroups in hyperbolic groups).
In this paper, we consider stably hyperbolic groups. A Gromov hyperbolic group is \emph{stably hyperbolic} if for any $\phi\in \edo(G)$ and any $i\in\N$, there exists a $j>i$ such that $\phi^j(G)$ is hyperbolic. It seems quite posssible that all hyperbolic groups are actually stably hyperbolic \cite{OT}.

\begin{thm}\label{Main theorem 1}
Let $G=\ast_{i=1}^n G_i$ be a torsion-free stably hyperbolic group, and each factor $G_i$ with finite $\ufp$ degree $\duf(G_i)$.  Then for any endomorphism $\phi\in\edo(G)$,
$$\rk(\fix(\phi))\leq \frac{1}{4}\ell(\rk(G)+1)^2,$$
where the number $\ell=\max_{i=1}^n \duf(G_i)$, the maximal one of $\duf(G_i)$, $i=1,\ldots, n.$
\end{thm}

As a corollary, we have

\begin{thm}\label{main thm 2}
Let $G=\ast_{i=1}^t G_i\ast F_s$ be a free product, where $F_s$ is a free group of rank $s$, and each factor $G_i$ is a surface group. Then
\begin{enumerate}[leftmargin=*]
\item for any endomorphism $\phi\in\edo(G)$, the fixed subgroup has rank
$$\rk(\fix(\phi))\leq \frac{1}{4}(\rk(G)+1)^2,$$
\item for any $\phi\in\mon(G)$, we have $\rk(\fix(\phi))\leq (s+t)(\rk(G)-s-t+1)$. In particular, if $s=0$ and all the surface groups $G_i$ share the same rank, then $$\rk(\fix(\phi))\leq \rk(G).$$
\end{enumerate}
\end{thm}

For a graph group $G$ (i.e. right angled Artin group), Rodaro, Silva and Sykiotis \cite{RSS13} showed: $G$ has $\fgfpe$ if and only if it is a free product of finitely many free abelian groups of finite ranks.
Now, we give explicit bounds for the fixed subgroups of free products of free abelian groups.

\begin{thm}\label{main thm 3}\label{free product of free abel}
Let $G=\ast_{i=1}^n \Z^{t_i}$ be a free product of free abelian groups $\Z^{t_i}$ of rank $t_i$. Then for any endomorphism $\phi\in \edo(G)$, we have
$$\rk(\fix(\phi))\leq n(\rk(G)-n+1).$$
In particular, if the ranks $t_1=t_2=\ldots =t_n$, then $\rk(\fix(\phi))\leq \rk(G)=\sum_{i=1}^n t_i.$
\end{thm}

Fixed subgroups of monomorphisms of free products are known regarding Kurosh rank. In this paper, by introducing the concept of UFP degree, we extend Kurosh rank to actual rank, and consider general endomorphisms of free products. Furthermore, combining some results of O'Niell-Turner, Sela and Sykiotis, we get a explicit bounds for the fixed subgroups of stably hyperbolic groups.
The paper is organized as follows. In Section \ref{Section. 2}, we review some useful facts on free products, especially Sykiotis' work on fixed subgroups of symmetric endomorphisms. In Section \ref{Section. 3}, we introduce some concepts on fixed subgroups to quantitatively analysis the ranks. At last, in Section \ref{Section. 4}, we give proofs of the main results, and show some examples.

\section{Preliminaries}\label{Section. 2}

In papers \cite{Sy02, Sy07}, Sykiotis studied the fixed subgroups of symmetric endomorphisms of free products of groups. For later use, let us first review some important definitions and facts in this section.

\subsection{Rank and Kurosh rank}
Let $G$ be a finitely generated group and let $H$ be a nontrivial subgroup of $G$.
The \emph{rank} of $G$ denoted $\rk(G)$ is the minimal number of generators of $G$. A famous result of Grushko showed that, $G$ can be represented as a free product of freely indecomposable factors,  $G=\ast_{i=1}^n G_i$, and the set of factors (and hence the number $n$ of the factors) is well-defined up to isomorphism. Then $n$ is said to be the (absolute) \emph{Kurosh rank} of $G$, denoted $\krk(G)$.
By the Kurosh subgroup theorem, the subgroup $H$ is a free product
$$H =\ast _{j=1}^t H_j \ast F_s,$$
where $F_s$ is a free group of rank $s$ and every factor $H_j$ is the intersection of $H$ with a conjugate of some factor $G_i$ (either of $s$ or $t$ could be $0$ or infinity). If the representation of $G$ as a free product is changed by an isomorphism, then these intersections change, but the number $s+t$ is invariant. We say $s+t$ the \emph{Kurosh rank of $H$ in} $G$. denoted $\krk_G(H)$.

Note that Grushko's theorem states that the rank of groups is additive under free products, i.e., $\rk(\ast_{i=1}^n G_i)=\sum_{i=1}^n \rk(G_i).$ Thus
$$\krk_G(H)\leq \krk(H)\leq \rk(H).$$
The first equality holds if $s=0$ and none of $H_j$ can be split as a nontrivial free products, and the second equality holds if each $H_j$ is cyclic.

\subsection{Stable image}
To study the fixed subgroups of endomorphisms of free groups, Imrich and Turner \cite{IT89} introduced the following definition.

\begin{defn}[\cite{IT89}]
For a group $G$ and an endomorphism $\phi\in\edo(G)$, the $stable~ image$ $\phi^{\infty}(G)$ of $\phi$ is the intersection $$\phi^{\infty}(G):=\bigcap_{n=1}^{\infty}\phi^n(G).$$
\end{defn}
Clearly, the fixed subgroup $\fix(\phi)=\fix(\phi_{\infty})\leq \phi^{\infty}(G)$, where $\phi_{\infty}: \phi^{\infty}(G)\to \phi^{\infty}(G)$ is the restriction of $\phi$ to the stable image $\phi^{\infty}(G)$.

For a free group $F_n$, Imrich and Turner proved that for arbitrary endomorphism $\phi\in\edo(F_n)$, $\phi_{\infty}\in \aut(\phi^{\infty}(F_n))$. Here $\phi^{\infty}(F_n)$ is a subgroup of $F_n$, so it is a free group. They also proved that $\rk(\phi^{\infty}(F_n))\leq \rk(F_n)=n$. Then the Bestvina-Handel theorem (Scott conjecture) implies that $\rk(\fix(\phi))=\rk(\fix(\phi_{\infty}))\leq \rk(\phi^{\infty}(F_n))\leq n$.

\subsection{Symmetric endomorphism}

Let $G=\ast_{i=1}^n G_i$ and $K=\ast_{i=1}^m K_i$ be two free products (the factors $G_i$ and $K_i$ may be freely decomposable). Sykiotis \cite{Sy07} gave the following definition.

\begin{defn}[\cite{Sy07}]
A homomorphism $\phi: G\rightarrow K$ is said to be \emph{symmetric (with respect to the given splitting)} if each non-infinite-cyclic free factor of $G$ is mapped by $\phi$ into a conjugate of some non-infinite-cyclic free factor of $K$.
\end{defn}

It is easy to see that if each factor $G_i$ is freely indecomposable, then each monomorphism is symmetric. Since surface groups are freely indecomposable, we have

\begin{lem}\label{symmtric of surface product}
Let $G=\ast_{i=1}^t G_i\ast F_s$ be a free product, where $F_s$ is a free group of rank $s$, and each factor $G_i$ is a surface group. Then every monomorphism $\phi\in\mon(G)$ is symmetric.
\end{lem}

For a free product $G=\ast_{i=1}^n G_i$ of freely indecomposable factors, Collins and Turner \cite{CT94} studied the fixed subgroups of an automorphism $\phi\in\aut(G)$, and showed that the Kurosh rank $\krk_G(\fix(\phi))$ of the fixed subgroup can not exceed $\krk(G)=n$. Sykiotis \cite{Sy07} extended Collins and Turner's result to monomorphisms. Since every group can be represented as a free product of freely indecomposable factors, we translate \cite[Corollary 4]{Sy07} to the following.

\begin{thm}[Sykiotis, \cite{Sy07}]\label{krk inj}
Let $G$ be a group and $\phi\in\mon(G)$ a monomorphism. Then the Kurosh rank
$\krk_G(\fix(\phi))\leq \krk(G).$
\end{thm}

In general, the fixed subgroups of endomorphisms of a free product is more complicated than that of monomorphisms. But for some special groups, we can get some bounds for the fixed subgroups. For later use, we list the following result on free products of nilpotent groups.

\begin{thm}\cite[Theorem 7]{Sy07}\label{fixed subgp in nilpotent gp}
Let $G=\ast_{i=1}^n G_i$ be a free product of finitely generated nilpotent and finite groups. If $\phi\in\edo(G)$ is an endomorphism of $G$, then the fixed subgroup $\fix(\phi)$ of $\phi$ has Kurosh rank at most $n$.
\end{thm}

By using Bass-Serre theory, combining Theorem \ref{krk inj} and a structure theorem \cite[Theorem 1.2]{Sy02} of the fixed subgroup of symmetric endomorphisms, we have

\begin{thm}\label{structure theorem}
Let $G=\ast_{i=1}^nG_i$ be a free product of freely indecomposable factors. Then for any monomorphism $\phi\in\mon(G)$,
$$\fix(\phi)=\ast_{j=1}^t \fix(\phi|_{g_jG_{\sigma(j)}g_{j}^{-1}})\ast F_s,$$
where $\krk_G(\fix(\phi))=s+t\leq n$, and $\phi|_{g_jG_{\sigma(j)}g_{j}^{-1}}:{g_jG_{\sigma(j)}g_{j}^{-1}}\to {g_jG_{\sigma(j)}g_{j}^{-1}}$ is the restriction of $\phi$  to a conjugate of some factor $G_{\sigma(j)}$.
\end{thm}

\begin{proof}
Since $G=\ast_{i=1}^nG_i$ is a free product, $G$ acts on a tree $X$ with finite quotient graph and trivial edge groups. Moreover, the stabilizer $G_v$ of a vertex $v$ is a conjugate of some factor $G_i$. Then the conclusion of Theorem \ref{structure theorem} follows from Theorem \ref{krk inj} and \cite[Theorem 1.2]{Sy02} clearly.
\end{proof}

\section{Fixed subgroups of monomorphisms}\label{Section. 3}

In this section, we first introduce two properties on fixed subgroups, and then show some results on the fixed subgroups of monomorphisms of free products.

\subsection{$\fgfp$ property}

Recall that a group $G$ has $\fgfpm$ (resp. $\fgfpa$, $\fgfpe$), if for any $f\in \mon(G)$ (resp. $\aut(G)$, $\edo(G)$), the fixed subgroup $\fix(f)$ is finitely generated. Clearly, if $G$ has $\fgfpe$, then it has $\fgfpm$ and hence has $\fgfpa$. Moreover, the property $\fgfpa$ is not heritable, for example, in \cite{ZVW}, the authors showed that the group $F_n\times F_m$ for $m,n>1$ has $\fgfpa$ but its subgroup $F_2\times \Z$ does not. To quantitatively analysis the ranks of the fixed subgroups, we introduce

\begin{defn}\label{def for UFP}
Let $G$ be a finitely generated group, and let $\mon(G)$ denote the monoid of monomorphisms of $G$.

\begin{enumerate}[leftmargin=*]
\item $G$ is said to have $k$-$\fgfp$, if for any $\phi\in\mon(G)$, $\rk(\fix(\phi))\leq k\cdot \rk(G).$
The minimal number $k$ satisfying the above equation is said to be the \emph{$\fp$ degree} for the group $G$, denoted $\df(G)$. Namely,
$$\df(G):=\sup\left\{\frac{\rk(\fix(\phi))}{\rk(G)}~\mid~\phi\in\mon(G)\right\}\in [1, ~+\infty].$$

\item  $G$ is said to have $k$-$\ufgfp$ (``U" for uniformly), if for every finitely generated subgroup $H\leq G$ and any $\phi\in\mon(H)$,
    $\rk(\fix(\phi))\leq k\cdot \rk(H).$
The minimal number $k$ satisfying the above equation is said to be the \emph{$\ufp$ degree} for the group $G$, denoted $\duf(G)$. Namely,
$$\duf(G):=\sup\left\{\frac{\rk(\fix(\phi))}{\rk(H)}~\mid~H\leq G, ~\phi\in\mon(H)\right\}\in [1, ~+\infty].$$
\end{enumerate}
\end{defn}

Note that in the above definitions, we only consider monomorphisms, and omit the cases of endomorphisms and automorphisms because that are similar.  Clearly, $k$-$\ufgfp$ implies $k$-$\fgfp$, and many kinds of groups have $k$-$\fgfp$. Furthermore,  the property $k$-$\ufgfp$ is heritable. In precisely,

\begin{lem}\label{uf index}
Let $G$ be a finitely generated group and let $H\leq G$ be a finitely generated subgroup.
\begin{enumerate}[leftmargin=*]
\item $1\leq \df(G)\leq \duf(G)\leq \infty$;

\item If $G$ has $k$-$\ufgfp$, then every subgroup $H\leq G$ also has $k$-$\ufgfp$ and hence has $k$-$\fgfp$, i.e.,
$$1\leq \df(H)\leq \duf(H)\leq \duf(G)\leq k.$$

\item $\df(G)=\duf(G)=1$ if $G$ is one of the following,

(a) ~free abelian groups $\Z^n$, \quad (b) ~free groups $F_n$, \quad (c) ~surfaces groups.

\item $\df(F_2\times \Z)=\infty$, and hence $\duf(G)=\infty$ if $G$ contains a subgroup that is isomorphic to $F_2\times \Z$.
\end{enumerate}
\end{lem}

\begin{proof}
Items (1) and (2) are trivial. Item (3) follows from Theorem \ref{Scott conj for free gp endo} and Theorem \ref{scott conj for surface gp} clearly.
To prove item (4), it suffices to show that $\fix(f)$ is not finitely generated for some $f\in \mon(F_2\times \Z)$. Indeed,
let
$$F_2\times \Z=\langle a,b,t \mid [a,t], [b,t] \rangle,$$
and let $f:F_2\times \Z\to F_2\times \Z$  such that $a\mapsto at$, $b\mapsto b$, $t\mapsto t$.
Then an element $u\in \fix(f)$ if and only if it has zero exponent sum in $a$.
So $\fix(f)$ is isomorphic to $F_{\infty}\times \Z$ generated by the infinite set $\{t, ~a^iba^{-i}|i\in\Z\}$.
\end{proof}

\subsection{$\fgfp$ in free products}\label{subsect. key prop}

In this subsection, we will show that $\fgfpm$ and $\fgfpa$ are preserved under free products.
To give the proof, we introduce two key propositions.

\begin{prop}\label{Main Lemma}
Let $G=\ast_{i=1}^n G_i$ be a free product, where each factor $G_i$ is a freely indecomposable group satisfying  $\rk\fix(f)\leq k_i\cdot \rk(G_i)$ for any $f\in \mon(G_i)$. Then for any monomorphism $\phi\in\mon(G)$, we have
$$\rk(\fix(\phi))\leq n(\max_{i=1}^n k_i)(\rk(G)-n+1),$$
while all the factors $G_i$ share the same rank, we have $\rk(\fix(\phi))\leq (\max_{i=1}^n k_i)\rk(G).$
\end{prop}

Before give the proof, we have a direct corollary on the $\fp$ degree of fixed subgroups.

\begin{cor}\label{Main Lemma on degree}
Let $G=\ast_{i=1}^n G_i$, where each $G_i$ is a freely indecomposable group with finite $\fp$ degree $\df(G_i)$. Then
$$\df(G)\leq n\cdot\max_{i=1}^n \df(G_i),$$
while all the factors $G_i$ have the same rank, then $\df(G)\leq \max_{i=1}^n \df(G_i).$
\end{cor}

\begin{proof}[\textbf{Proof of Proposition \ref{Main Lemma}}]
 According to Theorem \ref{structure theorem}, $\fix(\phi)=\ast_{i=1}^sH_i\ast F_t$, where $s+t\leq n$ and each $H_i=\fix(\phi|_{g_iG_{\sigma(i)}g_i^{-1}})$ is the fixed subgroup of a conjugate of some $G_{\sigma(i)}$. Since $G_i$'s have the property $\rk\fix(f)\leq k_i\cdot \rk(G_i)$ for any $f\in \mon(G_i)$,  we have
 $$\rk(H_i)\leq k_{\sigma(i)}\cdot \rk(g_iG_{\sigma(i)}g_i^{-1})=k_{\sigma(i)}\cdot\rk(G_{\sigma(i)})\leq k_{\sigma(i)}\cdot (\rk(G)-n+1).$$
 It follows
\begin{eqnarray}
\rk(\fix(\phi))&=& t+\sum_{i=1}^s \rk(H_i) \leq t+\sum_{i=1}^s k_{\sigma(i)}\cdot \rk(G_{\sigma(i)})\notag\\
               &\leq& n(\max_{i=1}^n k_i)(\rk(G)-n+1).
\end{eqnarray}
In particular, if all the factors $G_i$ have the same rank $\rk(G_1)=\cdots=\rk(G_n)$, then
$$\rk(\fix(\phi)) \leq t+\sum_{i=1}^s k_{\sigma(i)}\cdot \rk(G_{\sigma(i)})
               \leq (\max_{i=1}^n k_i)\cdot\rk(G).$$
\end{proof}

\rem Note that in the previous proof, $G_{\sigma(i)}$ may equal to $G_{\sigma(j)}$ for distinct $i\neq j$, see Section \ref{Section. 4} for an example.

\begin{prop}\label{key prop}
Let $G=\ast_{i=1}^n G_i$ be a free product, where each factor is a (not necessarily freely indecomposable) group $G_i$ with finite UFP degree $\duf(G_i)$. Then for any subgroup $H\leq G$ and any monomorphism $\phi\in\mon(H)$,
$$\rk(\fix(\phi))\leq \frac{1}{4}\ell (\rk(H)+1)^2,$$
where $\ell=\max_{i=1}^n\duf(G_i)$.
\end{prop}

\begin{proof}
 Let $H\leq G$ be a finitely generated subgroup and $\phi\in \mon(H)$ a monomorphism. Split $H$ as a free product of freely indecomposable factors, $H=\ast_{j=1}^sH_j$, where the absolute Kurosh rank (other than the Kurosh rank in $G$) $\krk(H)=s\leq\rk(H)$, and each factor $H_j$ is either a freely indecomposable group contained in a conjugate of a factor $G_{i_j}$ or the free cyclic group $\Z$ (with $\duf(\Z)=1$ clearly). Then Lemma \ref{uf index} implies
 $$\duf(H_j)\leq \duf(G_{i_j})\leq \max_{i=1}^n\duf(G_i)=\ell,$$
and hence $\rk\fix(f)\leq \ell\cdot\rk(H_j)$  for any $f\in \mon(H_j)$. Applying Proposition \ref{Main Lemma} to the monomorphism
 $$\phi: ~ H\to H=\ast_{j=1}^sH_j,$$
 we have $\rk(\fix(\phi))\leq \ell s(\rk(H)-s+1)\leq  \frac{1}{4}\ell(\rk(H)+1)^2.$
\end{proof}

\begin{thm}\label{main thm 0}
A free product $\ast_{i=1}^n G_i$ has $\fgfpm$ (resp. $\fgfpa$) if and only if the factor groups $G_1, G_2, \ldots, G_n$ all have $\fgfpm$ (resp. $\fgfpa$).
\end{thm}

\begin{proof}
Since the proofs of the two cases  $\fgfpm$ and $\fgfpa$ are parallel, we only consider the case $\fgfpm$, and leave the other case $\fgfpa$ for the reader.

Ar first, we prove the ``only if" part. Without loss of generality, suppose $G_1$ dose not have $\fgfpm$, i.e., there is a monomorphism $f: G_1\to G_1$ such that $\fix(f)$ is not finitely generated. Let $\id_i\in\mon(G_i)$ be the identity of $G_i$. Then the free product $f*\id_2*\cdots*\id_n\in \mon(G)$. Clearly, the fixed subgroup
$$\fix(f*\id_2*\cdots*\id_n)=\fix(f)*G_2*\cdots *G_n$$
is not finitely generated, contradicting the hypothesis that $G$ has $\fgfpm$.

Now let us consider the ``if" part. There are two cases:

(i) All the factors $G_i$ are freely indecomposable. Then any $f\in\mon(G)$ maps each factor $G_i$ to a conjugate of some $G_j$, as in the proof of Proposition \ref{Main Lemma}, we have $\fix(f)=\ast_{i=1}^sH_i\ast F_t$, where $s+t\leq n$ and each $H_i=\fix(f|_{g_iG_{\sigma(i)}g_i^{-1}})$ is the fixed subgroup of a conjugate of some $G_{\sigma(i)}$. Since $G_i$ all have $\fgfpm$, the $H_i$ are all finitely generated and hence $\fix(f)$ is also finitely generated.

(ii) Some factors $G_i$ are freely decomposable. Then each $G_i$ can be decomposed into $G_i=G'_{i_1}\ast\cdots\ast G'_{i_j}$, where each factor $G'_{i_k}$ is freely indecomposable. Since $G_i$ has $\fgfpm$, the factors $G'_{i_k}$ all have $\fgfpm$ by the ``only if" part. Therefore, we have reduced case (ii) to case (i).
\end{proof}

\begin{cor}\label{hyper 3 mfd gp}
Let $M=\#_{i=1}^n M_i$ be a connected sum of finitely many hyperbolic 3-manifolds. Then the fundamental group $\pi_1(M)$ has $\fgfpm$ (and hence $\fgfpa$). More precisely, for any monomorphism $f\in \mon(\pi_1(M))$, we have $\rk\fix(f)<2n\cdot\rk\pi_1(M)$.
\end{cor}

\begin{proof}
Since each $M_i$ is a hyperbolic 3-manifold, the fundamental group $\pi_1(M_i)$ is co-Hopfian by \cite{BGHM10}, i.e., every $\phi\in \mon(\pi_1(M_i))$ is an automorphism. Then Theorem \ref{lin-wang hyper 3 mfd} implies that $\pi_1(M_i)$ is $2$-$\fgfp$. Note that the fundamental group $\pi_1(M)=\ast_{i=1}^n\pi_1(M_i)$, and each $\pi_1(M_i)$ is freely indecomposable because $M_i$ is a hyperbolic 3-manifold, then Corollary \ref{Main Lemma on degree} implies $\rk(\fix(f))<2n\cdot\rk\pi_1(M)$ for any $f\in\mon(\pi_1(M))$, and hence $\pi_1(M)$ has $\fgfp_m$.
\end{proof}

\section{Proofs of main results and some examples}\label{Section. 4}

In the last section, we first prove the main theorems, and then give some examples.

\subsection{Proofs of main results}
Note that subgroups of hyperbolic groups are limit groups, by a result of Sela \cite{Se09}, we have

\begin{prop}\label{aut on stable image}
Let $G$ be a torsion-free hyperbolic group, and $\phi\in\edo(G)$ an endomorphism. Then the restriction
$\phi_{\infty}: ~~\phi^{\infty}(G)\to \phi^{\infty}(G)$
is an automorphism, and the Kurosh rank
$$\krk_{\phi^{\infty}(G)}(\fix(\phi))\leq \krk(\phi^{\infty}(G)).$$
\end{prop}

\begin{proof}
Let $\phi_n=\phi|_{\phi^n(G)}: \phi^n(G)\to \phi^{n+1}(G)$ be the restriction of $\phi$ to $\phi^n(G)$.
First, we assume that $\ker(\phi_n)$ is non-trivial for all $n$. Since the subgroups $\phi^n(G)$ of $G$ are also quotients of $G$, we have a decreasing sequence of $G$-limit groups as in \cite[Theorem 1.12]{Se09}, but with infinitely many terms,
$$G>\phi(G)>\phi^2(G)>\phi^3(G)>\cdots,$$
we get a contradiction. So $\phi_N$ is injective for some $N$.
Below we show that $\phi_{\infty}$ is an automorphism.
Note that for any $g\in\phi^{\infty}(G)=\bigcap_{n=N}^{\infty}\phi^n(G),$
there exists $g_n\in G$ such that $\phi^n(g_n)=g$ for every $n\geq N$.
Let $c_n=\phi^n(g_{n+1})\in \phi^n(G)$. Then $\phi(c_n)=g$. Since $\phi_N$ is injective, we have $c_N=c_n\in \phi^n(G)$ for all $n\geq N$ and hence $c_N\in\phi^{\infty}(G)$. This gives surjectivity of $\phi_{\infty}$, and hence $\phi_{\infty}$ is an automorphism.

Note that $\fix(\phi)=\fix(\phi_{\infty})\leq \phi^{\infty}(G)$, by using Theorem \ref{krk inj}, we have $$\krk_{\phi^{\infty}(G)}(\fix(\phi))=\krk_{\phi^{\infty}(G)}(\fix(\phi_{\infty}))\leq \krk(\phi^{\infty}(G)).$$
The proof is completed.
\end{proof}

\begin{rem}
Although many hyperbolic groups have been showed to be residually finite, Gromov's famous question remains open: Is every hyperbolic group residually finite? Therefore, Proposition \ref{aut on stable image} is not a direct corollary of \cite[Lemma 5]{Sy07}: Let $\phi$ be an endomorphism of a f.g. residually finite group $G$. Then the restriction $\phi_{\infty}: \phi^{\infty}(G)\rightarrow \phi^{\infty}(G)$ of $\phi$ to $\phi^{\infty}(G)$ is a monomorphism.
\end{rem}

\begin{proof}[\textbf{Proof of Theorem \ref{Main theorem 1}}]
Since $G=\ast_{i=1}^n G_i$ is a torsion-free hyperbolic group, every factor $G_i$ is also hyperbolic and torsion-free.  Then Proposition \ref{aut on stable image} implies that the restriction $\phi_{\infty}=\phi|_{\phi^{\infty}(G)}:\phi^{\infty}(G)\to \phi^{\infty}(G)$
is an automorphism for any $\phi\in\edo(G)$. Furthermore, since $G$ is stably hyperbolic, by \cite[Proposition 2]{OT}, the stable image $\phi^{\infty}(G)$ is a free factor of $\phi^N(G)$ for some $N$. So
$$\rk(\phi^{\infty}(G))\leq \rk(\phi^N(G))\leq \rk(G),$$
the last inequality holds because $\phi^N(G))$ is a quotient of $G$. Note that  $\fix(\phi)=\fix(\phi_{\infty})$ and $\phi^{\infty}(G)\leq G$, applying Proposition \ref{key prop} to the automorphism $\phi_{\infty}:\phi^{\infty}(G)\to \phi^{\infty}(G)$, we have
$$\rk(\fix(\phi))=\rk(\fix(\phi_{\infty}))\leq  \frac{1}{4}\ell\cdot (\rk(\phi^{\infty}(G))+1)^2\leq \frac{1}{4}\ell(\rk(G)+1)^2,$$
where $\ell=\max_{i=1}^n \duf(G_i)$.
\end{proof}

\begin{proof}[\textbf{Proof of Theorem \ref{main thm 2}}]
Item (2) follows from Proposition \ref{Main Lemma} clearly.
Now we prove item (1). Let $G=\ast_{i=1}^t G_i\ast F_s$ be a free product, where $F_s$ is a free group of rank $s$, and each factor $G_i$ is a surface group. Then $G$ is torsion free and hyperbolic. Moreover, for any $n$, by Kurosh subgroup theorem, $\phi^n(G)$ is a free product of finitely many free and surface groups, and hence it is hyperbolic with rank $\rk(\phi^n(G))\leq \rk(G)$. It implies that $G$ is stably hyperbolic. Note that surface groups and free groups have $\ufp$ degree $\duf=1$ according to Lemma \ref{uf index}. Therefore, $G$ satisfies the hypothesis of Theorem \ref{Main theorem 1} with $\ufp$ degree $\duf(G_i)=1$, and hence the conclusion of item (1) holds.
\end{proof}

\begin{proof}[\textbf{Proof of Theorem \ref{main thm 3}}]
 Without loss of generality, let $G=\ast_{i=1}^n\Z^{t_i}$ with $t_1\geq t_2\geq \cdots \geq t_n$. Then For any $\phi\in\edo(G)$, according to Theorem \ref{fixed subgp in nilpotent gp}, the fixed subgroup has the form
$\fix(\phi)=\ast_{j=1}^m H_j$ with $\krk_G(\fix(\phi))=m\leq n$, where each $H_j$ is either $\Z$ or a subgroup of some conjugate of $\Z^{t_i}$ and hence $\rk(H_j)\leq \rk(\Z^{t_i})\leq t_1$. Thus
$$\rk(\fix(\phi))\leq nt_1\leq n(\rk(G)-n+1).$$
If $t_1=\ldots =t_n$, then the conclusion holds clearly.
\end{proof}

\subsection{Some examples}

The first example shows that the bound in Theorem \ref{main thm 3} is sharp.

\begin{exam}
Let $G=\Z\ast\Z^2=\langle t\rangle\ast\langle a,b \mid [a,b]\rangle$, and $\varphi\in \aut(G)$ such that
$$\varphi(t)=ta, \quad\varphi(a)=a, \quad\varphi(b)=b.$$
Then $\fix(\varphi)=\Z^2\ast t\Z^2t^{-1}$, and hence
$\rk(\fix(\phi))=2(\rk(G)-1)$.
\end{exam}

At last, we give an example for the fixed subgroup of a free product of surface and free groups.

\begin{exam}
 Let  $H=\pi_1(S_g)=\langle a_1, b_1, \ldots, a_{g-1}, b_{g-1}, a_g, b \mid [a_1, b_1]\cdots[a_{g-1}, b_{g-1}][a_g, b]\rangle$, the fundamental group of an orientable surface $S_g$ of genus $g\geq 2$, and the commutator $[a,b]=aba^{-1}b^{-1}$. Let $\psi: H\to H$ as follows,
  $$\psi(a_i)=a_i, ~i=1,2,\ldots,g; \quad\psi(b_i)=b_i, ~i=1,2,\ldots,g-1; \quad \psi(b)=ba_3.$$
 Then $\id\neq \psi\in\aut(H)$, and hence, by Theorem \ref{scott conj for surface gp}, $\rk\fix(\psi)<\rk(H)=2g.$ Indeed, it is clear that $\fix(\psi)=\langle a_1, b_1,\ldots, a_{g-1}, b_{g-1}, a_g \rangle\cong F_{2g-1}$.  Let $\Z\ast H=\langle t \rangle *H$, and
 $$\varphi :\Z\ast H\to \Z\ast H, \quad \quad \varphi(t)=tb\psi(b^{-1}),\quad \varphi|_{H}=\psi.$$
 Then $\varphi\in\aut(\Z\ast H)$. Note that for any $h\in H$,
 $$\varphi(tht^{-1})=tb\psi(b^{-1})\cdot\psi(h)\cdot(tb\psi(b^{-1}))^{-1}=t[i_b\comp \psi \comp i_{b^{-1}}(h)]t^{-1}\in tHt^{-1},$$
where $i_b: H\to H, ~h\mapsto bhb^{-1}$ is the inner automorphism induced by $b$. Therefore, $\varphi|_{tHt^{-1}}: tHt^{-1}\to tHt^{-1}$ is an automorphism, and $$tht^{-1}\in \fix(\varphi|_{tHt^{-1}})\Longleftrightarrow h\in\fix(i_b\comp \psi \comp i_{b^{-1}})\cong \fix(\psi).$$
It implies $\fix(\varphi|_{tHt^{-1}})\cong \fix(\psi)\cong F_{2g-1}$. Note that $\krk_G(\fix(\varphi))\leq \krk(\Z*H)=2$ according to Theorem \ref{structure theorem}, so
$$\fix(\varphi)=\fix(\psi)*\fix(\varphi|_{tHt^{-1}})\cong F_{4g-2},$$
and hence $$\rk(\Z*H)=2g+1<\rk(\fix(\varphi))=4g-2< 2(\rk(\Z*H)-1)=4g.$$
\end{exam}

\vspace{6pt}

\noindent\textbf{Acknowledgements.} The authors thank Ke Wang for helpful communications. The authors also thank the referee whose valuable and detailed comments helped to greatly improve this paper.

\end{document}